\documentclass[letterpaper, 10 pt, conference]{ieeeconf}  

\usepackage{etoolbox}
\newtoggle{ext}     
\toggletrue{ext}    

\usepackage{amsmath}
\usepackage{amssymb}
\usepackage{amsthm}

\usepackage[ruled,vlined,linesnumbered]{algorithm2e}
\usepackage{algpseudocode}

\usepackage{mathtools,cuted}
\usepackage{mathrsfs}

\usepackage{graphicx}
\usepackage{subcaption}
\usepackage[font=small]{caption}
\usepackage{hhline}
\usepackage{array}

\usepackage{enumerate}

\usepackage{stfloats}
\usepackage{siunitx}
\usepackage[usenames, dvipsnames]{color}
\usepackage{cite}
\usepackage{url}
\usepackage[linkcolor=blue,citecolor=blue,urlcolor=blue]{hyperref}
\usepackage[colorinlistoftodos]{todonotes}

\usepackage{gensymb}

\usepackage{paralist}
\usepackage{todonotes}

\newtheorem{theorem}{Theorem}[section]
\newtheorem{corollary}{Corollary}[section]
\newtheorem{lemma}{Lemma}[section]
\newtheorem{remark}{Remark}[section]

\usepackage{graphicx}

\usepackage{stackengine}

\IEEEoverridecommandlockouts                              
\renewcommand{\baselinestretch}{0.8928}     
\title{\LARGE \bf
GuSTO: Guaranteed Sequential Trajectory Optimization \\ via Sequential Convex Programming
}

\author{Riccardo Bonalli, Abhishek Cauligi, Andrew Bylard, Marco Pavone 
\thanks{R. Bonalli, A. Cauligi, A. Bylard, and M. Pavone are with the Department of Aeronautics and Astronautics, Stanford University, Stanford, CA 94305. \{\tt rbonalli, acauligi, bylard, pavone\} {\tt@stanford.edu}.}
\thanks{This work was supported in part by NASA under the Space Technology Research Program, NASA Space Technology Research Fellowship Grants NNX16AM78H and NNX15AP67H, Early Career Faculty Grant NNX12AQ43G and Early Stage Innovations Grant NNX16AD19G, and by KACST.}
}

\begin{document}

\maketitle
  \thispagestyle{empty}
\pagestyle{empty}

\begin{abstract}
Sequential Convex Programming (SCP) has recently seen a surge of interest as a tool for 
trajectory optimization. However, most available methods lack rigorous performance guarantees and they are often tailored to specific optimal control setups. In this paper, we present GuSTO (Guaranteed Sequential Trajectory Optimization), an algorithmic framework to solve trajectory  optimization problems for control-affine systems with drift. GuSTO generalizes earlier SCP-based methods for trajectory optimization (by addressing, for example, goal-set constraints and problems with either fixed or free final time) and enjoys theoretical convergence guarantees in terms of convergence to, at least, a stationary point. The theoretical analysis is further leveraged to devise an accelerated implementation of GuSTO, which originally infuses ideas from indirect optimal control into an SCP context. Numerical experiments on a variety of trajectory optimization setups show that GuSTO generally outperforms current state-of-the-art approaches in terms of success rates, solution quality, and computation times.
\end{abstract}

\section{Introduction}
\label{sec:introduction}

Trajectory optimization algorithms play a key role in robot motion planning, either being applied directly to solve planning problems or being used to refine coarse trajectories generated by other methods. A wide variety of algorithmic frameworks have been proposed \cite{SchulmanDuanEtAl2014,LiuLu2014,MaoSzmukEtAl2016,DinhDiehl2010,VerscheureDemeulenaereEtAl2009, RatliffZuckerEtAl2009,KalakrishnanChittaEtAl2011,Betts1998}, 
and though they have had success on a broad class of robotic systems, a large gap remains in establishing practical guidelines for applying trajectory optimization to new systems and problem setups, placing guarantees on their behavior, and fully exploiting optimal control theory to improve performance.

In particular, additional work is required to achieve more general, well-analyzed frameworks for trajectory optimization algorithms which meet the following key desiderata:

\begin{enumerate}
\item {\em High computational speed:} Even on high-dimensional systems having complex dynamics and constraints, trajectory optimization algorithms should converge rapidly, allowing quick responses to commands and rapid replanning in uncertain or changing environments.

\item {\em Theoretical guarantees:} A reliable framework hinges on strong theoretical guarantees. Specifically, trajectory optimization algorithms should (i) guarantee initialization-independent convergence to, at least, a stationary point, (ii) ensure hard enforcement of dynamical constraints, especially as many robotic systems are nonholonomic, and (iii) provide that these guarantees are discretization-independent, since some robotic systems may call for specific numerical schemes.

\item {\em Generality:} Trajectory optimization frameworks should be broadly applicable to different robot motion planning problems, including involving complex robotic systems (e.g., nonconvex, nonholonomic dynamics, drift systems, etc.), flexible problem setups (e.g., free final time, goal sets, etc.), and diverse initialization strategies.

\end{enumerate}

\begin{figure}[t!]
\centering
\includegraphics[width=0.95\columnwidth]{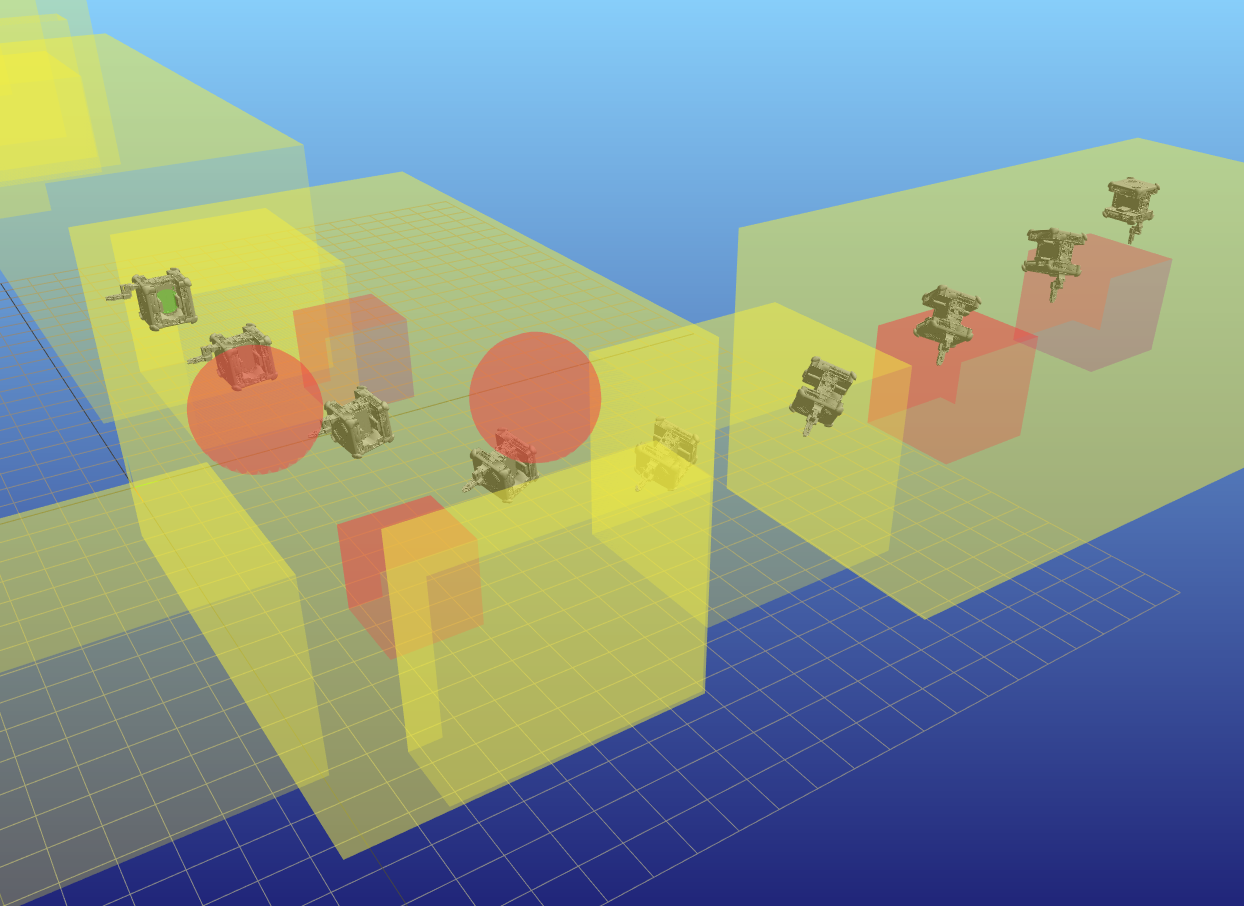}
\caption{GuSTO used to generate a dynamically-feasible, collision-free trajectory for the Astrobee free-flying spacecraft robot using a simple straight-line initialization \cite{SmithBarlowEtAl2016}.}
\label{fig:free_flyer_mp_scp}
\end{figure}

{\em Related work:} The trajectory optimization spectrum can be divided into global search methods and local methods. Global search methods include motion planning techniques, such as  asymptotically optimal sampling-based motion planning (SBP) algorithms (e.g., RRT$^*$, PRM$^*$, and FMT$^*$) \cite{LaValleKuffner2000,KavrakiSvestkaEtAl1996,JansonSchmerlingEtAl2015}. Though these require no initialization, they scale poorly to high-dimensional systems with kinodynamic constraints. For such systems, SBP techniques require enormous computational time and are thus instead used in practice to initialize other trajectory optimization algorithms.

Local methods include indirect methods, in particular including shooting methods \cite{Betts1998}. Built on an efficient coupling of necessary conditions of optimality, such as the Pontryagin Maximum Principle \cite{Pontryagin1987}, and Newton's methods, these have the fastest convergence rate, but they are highly sensitive to initialization and are thus difficult to apply to different tasks.
Another class of efficient local procedures is direct methods. One of these, which is the focus of this paper, is sequential convex programming (SCP), a framework which has been quite successful in the robotics community \cite{AugugliaroSchoelligEtAl2012,SchulmanDuanEtAl2014,LiuLu2014,MaoSzmukEtAl2016,MorganChungEtAl2014,Virgili-llopZagarisEtAl2017}. SCP successively convexifies the costs and constraints of a nonconvex optimal control problem, seeking a solution to the original problem through a series of convex problems \cite{FleuryBraibant1986,BoydVandenberghe2004}. Examples include TrajOpt \cite{SchulmanDuanEtAl2014}, Liu, et al. \cite{LiuLu2014}, and Mao, et al. \cite{MaoSzmukEtAl2016}. However, these suffer a number of deficiencies, as summarized in Table \ref{table:Approaches}. For example, \mbox{Traj}Opt provides high speed and broad applicability to robotic systems, but the penalization of dynamical constraints and the missing development of convergence guarantees preclude exact feasibility and numerical robustness, respectively. Similarly, \cite{LiuLu2014} only holds for a particular time-discretization, and though the approach in \cite{MaoSzmukEtAl2016} is discretization-indepedent, it cannot ensure hard enforcement of dynamics. Further, its convergence analysis relies on complex Lagrange multipliers from which it is difficult to extract numerically useful information, a key capability exploited in part in our work. Finally, in most of these works, extensions to free final time and goal-set constraints are not addressed. 

One last family of widespread procedures in trajectory optimization is variational methods. These include deterministic covariant approaches such as CHOMP \cite{RatliffZuckerEtAl2009} and probabilistic gradient descent approaches such as STOMP \cite{KalakrishnanChittaEtAl2011}. Similar to SCP, these do not necessarily require high-quality initializations, but theoretical guarantees are not easy to provide. Indeed, CHOMP does not incorporate the dynamical evolution of a system, while STOMP can only account for constraints through direct penalization, preventing hard enforcement of dynamics. Thus, convergence guarantees are not provided in both approaches.

\newcolumntype{C}{>{\centering\arraybackslash}p{1.65cm}}
\newcommand{\Centerstackmod}[1]{\addstackgap[3pt]{\Centerstack{#1}}}
\setstackgap{S}{1pt}

\begingroup
\begin{table*}[t]
\scriptsize
\centering

\hspace{3pt}\begin{tabular}{ C | C | C | C | C | C | C |}
\cline{2-7}
& \Centerstackmod{Hard enforcement {of dynamics} constraints}
& \Centerstackmod{Continuous-time convergence guarantees}
& \Centerstackmod{Independent {of time} discretization}
& \Centerstackmod{Free {final time}}
& \Centerstackmod{{Goal-set} constraint}
& \Centerstackmod{Provides {dual solution that} {can warm-start} {shooting method}} \\ \hhline{~|======|}
\end{tabular}

\begin{tabular}{| C || C | C | C | C | C | C |}
\hline
\addstackgap[3pt]{Optimal SBP} & $\bullet$ & $\bullet$ & $\bullet$ & $\bullet$ & $\bullet$ &  \\
\hline
\addstackgap[3pt]{TrajOpt \cite{SchulmanDuanEtAl2014}} & & & & & & \\
\hline
\addstackgap[3pt]{Liu, et al. \cite{LiuLu2014}} & $\bullet$ & & & $\bullet$ & $\bullet$ & \\
\hline
\addstackgap[3pt]{Mao, et al. \cite{MaoSzmukEtAl2016}} & & $\bullet$ & $\bullet$ & & & \\
\hline
\addstackgap[3pt]{CHOMP \cite{RatliffZuckerEtAl2009}} & & & $\bullet$ & & & \\
\hline
\addstackgap[3pt]{STOMP \cite{KalakrishnanChittaEtAl2011}} & & & & & & \\
\hline
\addstackgap[3pt]{This Work} & $\bullet$ & $\bullet$ & $\bullet$ & $\bullet$ & $\bullet$ & $\bullet$  \\

\hline
\end{tabular}
\caption{Comparison with existing trajectory optimization schemes.}
\label{table:Approaches}
\vspace{-15pt}
\end{table*}
\endgroup

{\em Statement of Contributions:} To begin to fill these gaps, our main contributions in this paper are as follows: First, we introduce Guaranteed Sequential Trajectory Optimization (GuSTO), an SCP-based algorithmic framework for trajectory optimization. More precisely, we provide a generalized continuous-time SCP scheme applied to drift control-affine nonlinear dynamical systems subject to control and state constraints (including collision-avoidance) and goal-set constraints, guaranteeing dynamic feasibility with either fixed or free final time. Second, we provide a theoretical analysis for this framework, proving that that the limiting solution of our continuous-time scheme is a stationary point in the sense of the Pontryagin Maximum Principle \cite{Pontryagin1987}. This generalizes the work in \cite{MaoSzmukEtAl2016} for control-affine systems and introduces stronger theoretical guarantees than the current state-of-the-art. Moreover, the generality of our framework enables these guarantees to be independent of the chosen time discretization scheme and the method used to find a solution at each SCP iteration. Indeed, the framework is broadly applicable to many different robot motion planning and trajectory optimization problems, which then enjoy the same guarantees. This analysis is further leveraged to accelerate convergence by initializing shooting methods with the dual solutions of SCP iterations. Third, we provide practical guidelines based on our analysis, including proper handling of constraints and initialization strategies. Moreover, we demonstrate the framework through numerical and hardware experiments, provide comparison to other approaches, and provide a Julia library for our trajectory optimization framework.

To the best of our knowledge, our framework uniquely meets all three aforementioned desiderata, rapidly providing theoretically desirable trajectories for a broad range of robotic systems and problem setups (see Table \ref{table:Approaches} for a comparison to some existing approaches).

\section{Problem Formulation and Overview of SCP}
\label{sec:formulation}
We begin by reviewing the optimal control problem of interest in Section \ref{subsec:ocp} and then provide an overview of an SCP framework for trajectory optimization in \ref{subsec:scp}.

\subsection{Trajectory Optimization as an Optimal Control Problem}
\label{subsec:ocp}

Given a fixed initial point $\bar{x}_0 \in \mathbb{R}^n$ and a final goal set $M_f \subseteq \mathbb{R}^n$, for every final time $t_f > 0$, we model our dynamics as a drift control-affine system in $\mathbb{R}^n$ of the form
\begin{eqnarray} \label{ref:AffineDynamics}
\begin{cases}
\displaystyle \dot{x}(t) = f(x(t),u(t)) = f_0(x(t)) + \sum_{i=1}^{m} u^i(t) f_i(x(t))\medskip \\
x(0) = \bar{x}_0 \quad , \quad x(t_f) \in M_f \quad \textnormal{s.t.} \;\; \textnormal{dist}(\bar{x}_0,M_f) > 0
\end{cases}
\end{eqnarray}
where $f_i : \mathbb{R}^n \rightarrow \mathbb{R}^n$, $i = 0,\dots,m$ are $C^1$ vector fields, and $\textnormal{dist}(x,A) = \inf_{y \in A} \|x-y\|_2$.

In this context, we design trajectory optimization as an optimal control problem with penalized state constraints. More specifically, we consider the Optimal Control Problem (\textbf{OCP}) consisting of minimizing the integral cost
\begin{equation} \label{ref:Cost}
\begin{split}
J(t_f, \: &x,u) = \int_{0}^{t_f} f^0(x(t),u(t)) \; \mathrm{d}t = \\
& \int_{0}^{t_f} \left( \| u(t) \|^2_R + u(t) \cdot f^0(x(t)) + g(x(t))\right) \; \mathrm{d}t
\end{split}
\end{equation}
under dynamics \eqref{ref:AffineDynamics}, among all the controls $u \in L^{\infty}([0,t_f],\mathbb{R}^m)$ satisfying $u(t) \in U$ almost everywhere in $[0,t_f]$. Here, $f^0 : \mathbb{R}^n \rightarrow \mathbb{R}^m$, $g : \mathbb{R}^n \rightarrow \mathbb{R}$ are $C^1$, $\| \cdot \|_R$ represents the norm that is given by a constant positive-definite matrix $R \in \mathbb{R}^{m \times m}$, and $U \subseteq \mathbb{R}^m$ provides control constraints. The final time $t_f$ may be free or fixed, and hard enforcement of dynamical and goal-set constraints are naturally imposed by \eqref{ref:AffineDynamics}. Function $g = \bar g_1 + \omega \bar g_2$ sums up the contributions of the state-depending terms $\bar g_1$ of the cost unrelated to constraints and of the sum of state constraint violations $\bar g_2$ (e.g., collision-avoidance violation), where $\omega \ge 1$ is a penalization weight. We stress that penalizing state constraints is fundamental to obtaining theoretical guarantees in the sense of the classical Pontryagin Maximum Principle \cite{Pontryagin1987} (see Theorem \ref{ref:theoSCP}), stronger than standard Lagrange multiplier rules. However, in Sec. \ref{sec:algorithmguarantees}, we provide an algorithm under this formulation which can still enforce hard state constraints up to some chosen tolerance $\varepsilon \geq 0$. 

\subsection{Sequential Convex Programming}
\label{subsec:scp}
We proceed by applying sequential convex programming to solve our optimal control problem. Under the assumption that $U$ is convex, SCP consists of iteratively linearizing the nonlinear contributions of (\textbf{OCP}) around local solutions, thus recursively defining a sequence of simplified problems. More specifically, for a given $t^0_f > 0$, assume we have some continuous curve $x_0 : [0,t^0_f] \rightarrow \mathbb{R}^n$ and some control law $u_0 : [0,t^0_f] \rightarrow \mathbb{R}^m$, continuously extended in the interval $(0,+\infty)$. Defined inductively, at iteration $k+1$, the Linearized Optimal Control Problem (\textbf{LOCP})$_{k+1}$ consists of minimizing the new integral cost
\begingroup
\footnotesize
\begin{equation} \label{ref:Cost_k}
\begin{split}
& J_{k+1}(t_f, x, u) = \int_{0}^{t_f} f^0_{k+1}(t,x(t),u(t)) \; \mathrm{d}t = \\
& \int_{0}^{t_f} \left( \| u(t) \|^2_R + h_k\left( \| x(t) - x_k(t) \|^2 - \Delta_k \right) \right) \; \mathrm{d}t \; + \\
& \int_{0}^{t_f} u(t) \cdot \left( f^0(x_k(t)) + \frac{\partial f^0}{\partial x}(x_k(t)) \cdot (x(t) - x_k(t)) \right) \; \mathrm{d}t \; + \\
& \int_{0}^{t_f} \left( g_k(x_k(t)) + \frac{\partial g_k}{\partial x}(x_k(t)) \cdot (x(t) - x_k(t)) \right) \; \mathrm{d}t
\end{split}
\end{equation}
\endgroup
where $g_k = \bar g_1 + \omega_k \bar g_2$ and $h_k(s)$ is any smooth approximation of $\omega_k \max\{0,s\}$ \cite[Ch. 10]{Lee2003}, under the new dynamics
\begingroup
\footnotesize
\begin{eqnarray} \label{ref:AffineDynamics_k}
\begin{cases}
\dot{x}(t) = f_{k+1}(t,x(t),u(t)) = \medskip \\
\displaystyle \quad \left( f_0(x_k(t)) + \sum_{i=1}^{m} u^i(t) f_i(x_k(t)) \right) + \medskip \\
\displaystyle \quad \left( \frac{\partial f_0}{\partial x}(x_k(t)) + \sum_{i=1}^{m} u^i_k(t) \frac{\partial f_i}{\partial x}(x_k(t)) \right) \cdot (x(t) - x_k(t)) \medskip \\
x(0) = \bar{x}_0 \quad , \quad x(t_f) \in M_f
\end{cases}
\end{eqnarray}
\endgroup
coming from the linearization of nonlinear vector fields, among all controls $u \in L^{\infty}([0,t_f],\mathbb{R}^m)$ satisfying $u(t) \in U$ almost everywhere in $[0,t_f]$, where $(t^k_f,x_k,u_k)$ is a solution for the linearized problem at the previous iteration, i.e. (\textbf{LOCP})$_{k}$, continuously extended in the interval $(0,+\infty)$. As a result of the functions $h_k \in C^{\infty}$, we can provide trust-region-type constraints on state trajectories using uniformly bounded scalars $0 \le \Delta_k \le \Delta_0$ and weights $1 \le \omega_0 \le \omega_k \le \omega_{\max}$ (no such bounds are considered on controls because $u$ appears linearly), which at the same time penalize state constraints violations $\bar g_2$. Here, the user may make vary $\Delta_k$ and $\omega_k$ at each iteration --- these are used merely to ease the search for a solution of (\textbf{LOCP})$_{k+1}$. Problem (\textbf{LOCP})$_{1}$ is linearized around an initializing couple $(x_0,u_0)$, and this initialization curve should be as close as possible to a feasible or even optimal curve for (\textbf{LOCP})$_{1}$, although we do not require that $(x_0,u_0)$ is feasible for (\textbf{OCP}).

The sequence of problems (\textbf{LOCP})$_{k}$ is correctly defined if, for each iteration $k \ge 1$, an optimal solution for (\textbf{LOCP})$_{k}$ exists. For this, we consider the following assumptions:

\begin{itemize}
\item[$(A_1)$] The set $U$ is compact and convex, while the set $M_f$ is a compact submanifold (either with or without boundary).
\item[$(A_2)$] Mappings $f^0$, $g$, vector fields $f_i$, $i = 0,\dots,m$ and their differentials have compact supports.
\item[$(A_3)$] At every iteration $k \ge 1$, problem (\textbf{LOCP})$_{k}$ is feasible. Moreover, for free final time problems, there exists a constant $b > 0$ such that, every feasible tuple $(t_f,x,u)$ for (\textbf{LOCP})$_{k}$ satisfies $t_f \le b$, for every iteration $k \ge 1$.
\end{itemize}

\noindent Under these assumptions, classical existence Filippov-type arguments \cite{Filippov1962,LeeMarkus1967} show that at each iteration $k \ge 1$, the problem (\textbf{LOCP})$_{k}$ has at least one optimal solution. Here, some comments are in order. Assumption $(A_2)$ is not limiting and can be easily satisfied by multiplying all noncompliant maps by smooth cut-off functions having supports contained in the working space. Moreover, it is standard in control theory to assume time-bounded strategies, and we can satisfy $(A_3)$ by simply considering the notion of virtual control \cite{MaoSzmukEtAl2016}. Indeed, we stress the fact that most trajectory optimization applications effortlessly satisfy Assumptions $(A_1)$-$(A_3)$.

\section{GuSTO: Algorithm Overview\\ and Theoretical Analysis}
\label{sec:algorithmguarantees}
In this section, we present the algorithmic details for GuSTO in \ref{subsec:generalized_scp} and discuss its theoretical convergence guarantees to a stationary point in \ref{subsec:convergence_guarantees}.

\subsection{Generalized SCP Algorithm}
\label{subsec:generalized_scp}

SCP aims to solve (\textbf{OCP}) by iteratively seeking solutions to (\textbf{LOCP})$_{k}$. For this process, the progression through iterations must be designed carefully in order to achieve numerical efficiency and fast computation. Supported by classical approaches \cite{NocedalWright2006} and more recent results in SCP for robot trajectory optimization \cite{SchulmanDuanEtAl2014,MaoSzmukEtAl2016}, we propose a new general SCP scheme, named GuSTO (Guaranteed Sequential Trajectory Optimization), to solve (\textbf{OCP}), as reported in Algorithm \ref{ref:algoSCP}. Our main novelties are: 1) a time-continuous, broadly applicable setup ensuring convergence to a stationary point in the sense of the Pontryagin Maximum Principle \cite{Pontryagin1987} (see Corollary \ref{ref:corollSCP}), 2) hard enforcement of dynamical constraints and ease in considering free final time and goal-set problems, 3) a refined trust-region radius adaptation step based on a new model accuracy ratio which provides a definition of relative error between iterations and prevents the algorithm from becoming stuck in a cycle within its loops, 4) a theoretically justified stopping criterion based on closeness between iterated solutions. \vspace{3pt}

\SetKwInOut{Input}{Input}
\SetKwInOut{Output}{Output}
\SetKwInOut{Data}{Data}

\begin{algorithm}
\caption{GuSTO} \label{ref:algoSCP}
\Input{Trajectory $x_0$ and control $u_0$ defined in $(0,\infty)$.}
\Output{Solution $(x_k,u_k)$ for (\textbf{LOCP})$_{k}$ at iteration $k$.}
\Data{State constraints data $\Delta_0 > 0$, $\omega_0 \ge 1$, $\varepsilon \ge 0$;

\qquad \qquad Trust region scaling parameters $0 < \beta_{\textrm{fail}} < 1$,

\qquad \qquad $\beta_{\textrm{succ}} > 1$, $0<\rho_{0}<\rho_{1}<1$, $\gamma_{\textrm{fail}} > 1$.}

\Begin
{
    $k = 0$\\
    \While{$(t^{k+1}_f,x_{k+1},u_{k+1}) \neq (t^{k}_f,x_{k},u_{k})$ and $\omega_{k+1} \le \omega_{\max}$}
    {
      Solve (\textbf{LOCP})$_{k+1}$ for $(t^{k+1}_f,x_{k+1},u_{k+1})$\\
      \If{$\| x_{k+1} - x_{k} \|^2(\cdot) \le \Delta_k$}{
            Calculate model accuracy ratio $\rho_{(k)}$ in (\ref{ref:trust_region_performance})\\
            \If{$\rho_{(k)} > \rho_{1}$} {
                Reject solution $(t^{k+1}_f,x_{k+1},u_{k+1})$\\
                $\Delta_{k+1} \gets \beta_{\textrm{fail}}\Delta_k$ \ , \ $\omega_{k+1} \gets \omega_k$\\
            }
            \Else{
                Accept solution $(t^{k+1}_f,x_{k+1},u_{k+1})$\\
                \begingroup
                \small
                $\Delta_{k+1} \gets \begin{cases} \min\{\beta_{\textrm{succ}} \Delta_k , \Delta_0\} & \rho_{(k)}<\rho_{0}\\ \Delta_k & \rho_{(k)} \geq \rho_{0} \end{cases}$\\
                $\omega_{k+1} \gets \begin{cases} \omega_0 & \bar g_2(x_{k+1}(\cdot)) \le \varepsilon\\ \gamma_{fail} \omega_k & \bar g_2(x_{k+1}(\cdot)) > \varepsilon \end{cases}$\\
                \endgroup
            }
      }
      \Else{
            Reject solution $(t^{k+1}_f,x_{k+1},u_{k+1})$\\
            $\Delta_{k+1} \gets  \Delta_{k}$ \ , \ $\omega_{k+1} \gets \gamma_{fail} \omega_k$\\
        }
        $k \gets k+1$\\
    }
    \Return{$(t^{k}_f,x_{k},u_{k})$}
}
\end{algorithm}

Once (\textbf{LOCP})$_{k+1}$ is solved at some iteration $k$ (line \textbf{4}), we first check whether hard trust-region constraints are satisfied. In the positive case, we evaluate the ratio
\begingroup
\scriptsize
\begin{equation} \label{ref:trust_region_performance}
\begin{split}
& \rho_{(k)} = N_{k} / D_{k} = \bigg( | J(t^{k+1}_f,x_{k+1},u_{k+1}) - J_{k+1}(t^{k+1}_f,x_{k+1},u_{k+1}) | + \\
& \int_{0}^{t^{k+1}_f} \| f(x_{k+1}(t),u_{k+1}(t)) - f_{k+1}(t,x_{k+1}(t),u_{k+1}(t)) \| \; \mathrm{d}t \bigg) \Big/ \\
& \left( | J_{k+1}(t^{k+1}_f,x_{k+1},u_{k+1}) | + \int_{0}^{t^{k+1}_f} \| f_{k+1}(t,x_{k+1}(t),u_{k+1}(t)) \| \; \mathrm{d}t \right)
\end{split}
\end{equation}
\endgroup
which represents the relative error between the original cost/dynamics and their convexified versions. If this error is greater than some given tolerance, the linear approximation is too coarse and we reject the new solution, shrinking the trust region (lines \textbf{7}-\textbf{9}). Otherwise, we accept and update the trust region radius (lines \textbf{11}-\textbf{12} \cite{MaoSzmukEtAl2016}). Moreover, in the case that hard-penalized state constraints are not satisfied, we increase the value of the weight $\omega_k$ (line \textbf{13}), pushing the solver to seek constraint satisfaction (up to some threshold $\varepsilon \ge 0$) at the next iteration. On the other hand, when only soft trust-region constraints are satisfied, we increase the weight $\omega_k$ while maintaining the same radius $\Delta_k$ (lines \textbf{15}-\textbf{16}), pushing the solver to look for solutions that satisfy the trust-region constraints. The algorithm ends when successive iterations reach an identical solution or when the state constraint weight is greater than the maximum value $\omega_{\max}$.

\begin{remark} \label{ref:remarkAlgoSCP}
As a result of Assumptions $(A_1)$-$(A_3)$, we have $N_k \le \omega_k \tilde C \| x_{k+1} - x_{k} \|_{C^0}$, where $\tilde C \ge 0$ is some constant depending only on quantities defining (\textbf{OCP}). Moreover, since every solution $(t^{k}_f,x_{k},u_{k})$ satisfies the initial and final conditions for \eqref{ref:AffineDynamics}, it holds that $0 < \textnormal{dist}(\bar{x}_0,M_f) \le \| x_{k+1}(t^{k+1}_f) - x_{k+1}(0) \| \le D_k$. Therefore, since $\omega_k \le \omega_{\max}$, it is easily seen that Algorithm \ref{ref:algoSCP} never becomes stuck in the rejection step provided by lines \textbf{7}-\textbf{9}.
\end{remark}

No assumption on the initializing strategy $(x_0,u_0)$ is taken. From a practical point of view, this allows us to initialize GuSTO with simple, even infeasible, guesses for solutions of (\textbf{OCP}), such as a straight line in the state space. In this case, a suitable choice of the maximal value of the trust region radius $\Delta_0$ may be crucial to allow the method to correctly explore the space if the provided initialization is far from any optimal strategy. Finally, increasing the value of weights $\omega_k$ at line \textbf{3} of GuSTO eases the search for solutions satisfying state constraints up to the $\varepsilon$ tolerance.

\subsection{Theoretical Convergence Guarantees}
\label{subsec:convergence_guarantees}
In this section, we prove that GuSTO has the property of guaranteed convergence to an extremal solution. This result is achieved by leveraging techniques from indirect methods in a direct method context, which is a contribution of independent interest discussed further in Section \ref{subsec:shooting_methods}.

The convergence of GuSTO can be inferred by adopting one further regularity assumption concerning (\textbf{LOCP})$_{k}$:

\begin{itemize}
\item[$(A_4)$] At every iteration $k \ge 1$ of SCP, every optimal control $u_k$ of (\textbf{LOCP})$_{k}$ is continuous.
\end{itemize}

\noindent We stress that although Assumption $(A_4)$ seems limiting, many control systems in trajectory optimization applications naturally satisfy it \cite{ChitourJeanEtAl2008}. Moreover, the normality of Pontryagin extremals is sufficient (under minor assumptions) to ensure that $(A_4)$ holds \cite{ShvartsmanVinter2006}.

Thus, in view of the Pontryagin Maximum Principle \cite{Pontryagin1987}, our main theoretical result for SCP is the following:

\begin{theorem}[] \label{ref:theoSCP}
Suppose that Assumptions $(A_1)$-$(A_4)$ hold. Given any sequence of trust region radii and weights $((\Delta_k,\omega_k))_{k \in \mathbb{N}} \subseteq [0,\Delta_0] \times [\omega_0,\omega_{\max}]$, let $((t^k_f,x_k,u_k))_{k \in \mathbb{N}}$ be any sequence such that for every $k \ge 1$, $(x_k,u_k)$ is optimal for (\textbf{LOCP})$_{k}$ in $[0,t^k_f]$. Up to some subsequence:
\begin{itemize}
\item $t^k_f \rightarrow \tilde t_f \in [0,b]$, for the strong topology of $\mathbb{R}$
\item $x_k \hspace{-4pt} \rightarrow \tilde x \in C^0([0,\tilde t_f],\mathbb{R}^n)$, for the strong topology of $C^0$
\item $u_k \rightarrow \tilde u \in L^{\infty}([0,\tilde t_f],U)$, for the weak topology of $L^2$
\end{itemize}
as $k$ tends to infinity, such that, $(\tilde x,\tilde u)$ is feasible for (\textbf{OCP}) in $[0,\tilde t_f]$. Moreover, there exists a nontrivial couple $(\tilde p,\tilde p^0)$ such that the tuple $(\tilde x,\tilde p,\tilde p^0,\tilde u)$ represents a Pontryagin extremal for (\textbf{OCP}) in $[0,\tilde t_f]$. In particular, as $k$ tends to infinity, up to some subsequence:
\begin{itemize}
\item $(p_k,p^0_k) \rightarrow (\tilde p,\tilde p^0)$ for the strong topology of $C^0 \times \mathbb{R}$
\end{itemize}
where $(x_k,p_k,p^0_k,u_k)$ is a Pontryagin extremal of (\textbf{LOCP})$_{k}$.  Finally, for fixed final time $t_f$ problems, we have $\tilde t_f = t_f$. 
\end{theorem}

For sake of conciseness and continuity in the exposition, we report the proof of Theorem \ref{ref:theoSCP} in the \iftoggle{ext}{Appendix}{extended version of the paper \cite{BonalliCauligiEtAl2019}}.
The convergence of GuSTO to a stationary point, in the sense of the Pontryagin Maximum Principle, for (\textbf{OCP}) is quickly obtained as a corollary.

\begin{corollary} \label{ref:corollSCP}
Under $(A_1)$-$(A_4)$, in solving (\textbf{OCP}) by Algorithm \ref{ref:algoSCP}, only three mutually exclusive situations arise:
\begin{enumerate}
\item There exists an iteration $k \ge 1$ for which $\omega_k > \omega_{\max}$. Then, Algorithm \ref{ref:algoSCP} terminates, providing a solution for (\textbf{LOCP})$_{k}$ that does not satisfy state constraints.
\item There exists an iteration $k \ge 0$ for which $(t^{k+1}_f,x_{k+1},u_{k+1}) = (t^{k}_f,x_{k},u_{k})$. Then, Algorithm \ref{ref:algoSCP} terminates, returning a stationary point, in the sense of the Pontryagin Maximum Principle, for (\textbf{OCP}).
\item We have $(t^{k+1}_f,x_{k+1},u_{k+1}) \neq (t^{k}_f,x_{k},u_{k})$, for every iteration $k \ge 0$. Then, Algorithm \ref{ref:algoSCP} builds a sequence of optimal solutions for (\textbf{LOCP})$_{k}$ that has a subsequence converging (with respect to appropriate topologies) to a stationary point, in the sense of the Pontryagin Maximum Principle, for the original problem (\textbf{OCP}).
\end{enumerate}
\end{corollary}

\begin{proof}
\noindent Thanks to Remark \ref{ref:remarkAlgoSCP}, it is clear that only these three cases may happen and that they are mutually exclusive. Then, we only need to consider cases 2) and 3). The latter follows from Theorem \ref{ref:theoSCP}. If Algorithm \ref{ref:algoSCP} falls into case 2), then by applying the Pontryagin Maximum Principle \cite{Pontryagin1987} to (\textbf{LOCP})$_{k+1}$, we have that $(x_{k+1},p_{k+1},p^0_{k+1},u_{k+1})$ is the desired (Pontryagin) stationary point for (\textbf{OCP}).
\end{proof}

Case 1) of Corollary \ref{ref:corollSCP} represents a failure and means that we are not able to compute a feasible optimized strategy, either due to an infeasible problem or an infeasible initialization that could not be refined to feasibility. The same occurs when considering TrajOpt \cite{SchulmanDuanEtAl2014}.

On the other hand, both cases 2) and 3) represent success. However, it is important to remark that from a practical point of view, because of numerical errors, when we begin satisfying some convergence criterion on $((t^{k}_f,x_{k},u_{k}))_{k \in \mathbb{N}}$ (which is up to the user) while solving (\textbf{OCP}) by GuSTO, we usually fall into case 3) and rarely fall into case 2) of Corollary \ref{ref:corollSCP}. At this point, Theorem \ref{ref:theoSCP} becomes crucial to ensuring that we are actually converging to a stationary point, in the sense of the Pontryagin Maximum Principle, for (\textbf{OCP}). This holds for the whole sequence of solutions $(t^{k}_f,x_{k},u_{k})$, since it is itself a converging subsequence. In addition, these theoretical guarantees achieved in a continuous-time setting remain valid independent of the chosen discretization scheme, as sufficiently small time steps allow the discrete SCP solution to remain close to the solution of the original continuous-time problem (\textbf{OCP}).

Notice that a similar framework is considered in \cite{MaoSzmukEtAl2016}, in which, for an infinite number of iterations, one can only provide weak convergence up to some subsequence if the control constraint set $U \subseteq \mathbb{R}^m$ is convex and compact. Indeed, this last assumption does not imply that $L^{\infty}([0,t_f],U)$ is compact (e.g., take $U$ to be the closed unit ball \cite{Brezis2011}).
In any case, the result provided by Theorem \ref{ref:theoSCP} remains stronger because, unlike \cite{MaoSzmukEtAl2016}, we obtain strong convergence of both trajectories and Pontryagin extremals. This feature can be exploited to provide convergence acceleration, as demonstrated in the next section.

\begin{figure*}[t!]
\centering
  \begin{subfigure}[t]{0.49\textwidth}
  \centering
    \includegraphics[height=1.9in]{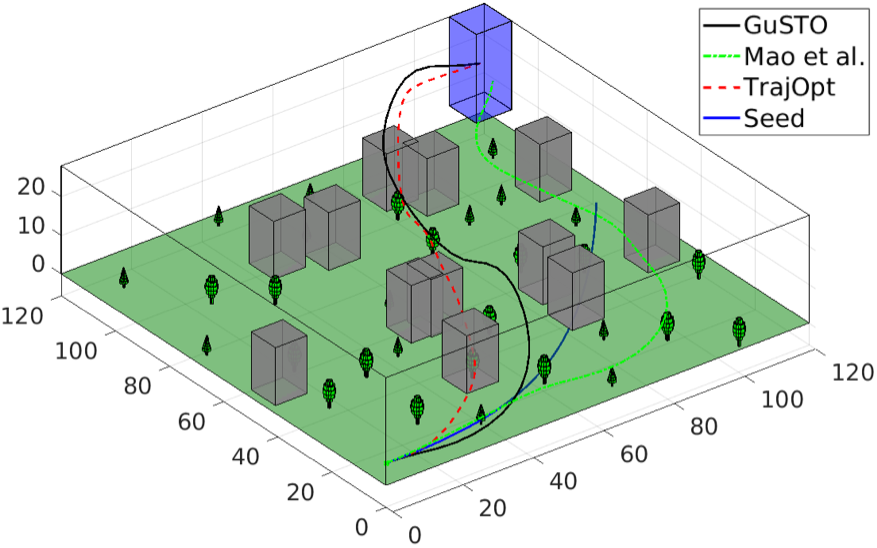}
    \caption{Using controller-tracked straight-line initialization}
    \label{fig:airplane_case2}
  \end{subfigure}
  \begin{subfigure}[t]{0.49\textwidth}
    \centering
    \includegraphics[height=1.9in]{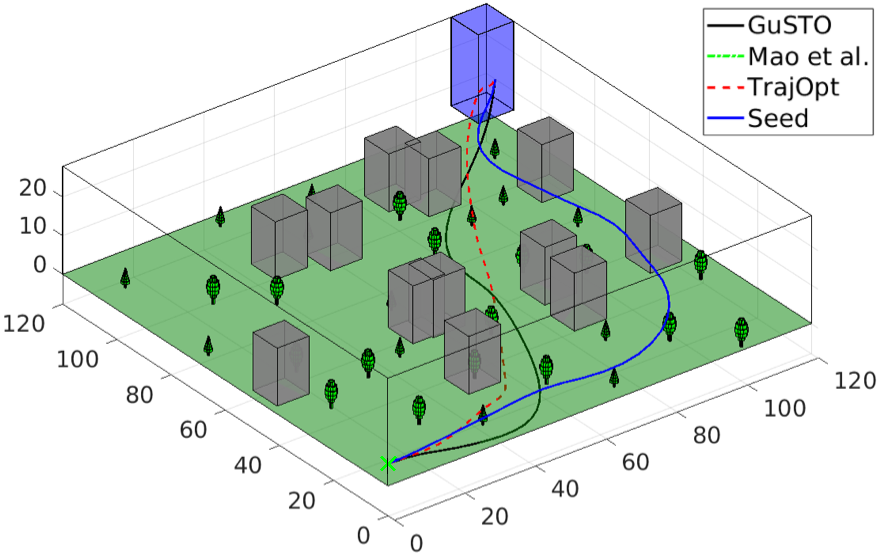}
    \caption{Using SOS planning initialization}
    \label{fig:airplane_case3}
  \end{subfigure}
  
  \caption{Comparing initialization strategies on an 8D airplane model for three different SCP algorithms.}
    \label{fig:airplane}
    \vspace{-15pt}
\end{figure*}

\subsection{Accelerating Convergence using Shooting Methods} \label{subsec:shooting_methods}
An important result provided by Theorem \ref{ref:theoSCP} is the convergence of Pontryagin extremals related to the sequence of solutions to problems (\textbf{LOCP})$_{k}$ towards a Pontryagin extremal related to the solution of (\textbf{OCP}) found by GuSTO. In particular, we can use this result to accelerate convergence by warm-starting shooting methods \cite{Betts1998} using the dual solution from each SCP iteration.

This can be shown as follows. Assuming that GuSTO is converging, the Lagrange multipliers $\lambda^0_k$ related to the initial condition $x(0) = \bar{x}_0$ for the finite dimensional discretization of problems (\textbf{LOCP})$_{k}$ approximate the initial values $p_k(0)$ of the adjoint vectors related to each (\textbf{LOCP})$_{k}$ \cite{GollmannKernEtAl2008}. Then, up to some subsequence, for every small $\delta > 0$, there exists an iteration $k_{\delta} \ge 1$ for which, for every iteration $k \ge k_{\delta}$, one has $\| \tilde p(0) - \lambda^0_k \| < \delta$, where $\tilde p$ is an adjoint vector related to the solution of (\textbf{OCP}) found by SCP. This means that, starting from some iteration $k \ge k_{\delta}$, we can run a shooting method to solve (\textbf{OCP}), initializing it by $\lambda^0_k$. Thus, at each iteration of GuSTO, we use the $\lambda^0_k$ provided by the solver to initialize the shooting method until convergence. In practice, this method provides a principled approach to facilitate fast convergence of sequential convex programming towards a more precise solution.


\section{Numerical Experiments and Discussion}
\label{sec:experiments}
In this section, we provide implementation details and examples to demonstrate various facets of our approach.

\subsection{Implementation Details}
\label{subsec:implementation_details}
We implemented the examples in this section in a trajectory optimization library written in Julia \cite{BezansonKarpinskiEtAl2012} and available at \url{https://github.com/StanfordASL/GuSTO.jl}. Computation times reported are from a Linux system equipped with a 4.3GHz processor and 32GB RAM.
For each system and compared algorithm, we discretized the cost and dynamics of the continuous-time optimal control problem using a trapezoidal rule, assuming a zero-order hold for the control. The number of discretization points $N$ was set to 30-40 in the presented results. We used the Bullet Physics engine to calculate signed distance fields used for obstacle avoidance constraints \cite{Coumans, Ericson2004}, and an SCP trial was marked as successful if the algorithm converged and the resulting solution was collision-free. In comparisons with Mao, et al. \cite{MaoSzmukEtAl2016}, since their approach does not address nonconvex state inequality constraints (e.g. collision avoidance), we chose to enforce linearized collision-avoidance constraints in their algorithm as hard inequality constraints.

\subsection{Batch Comparison using a Simple Initialization Scheme}
\label{subsec:scp_comparison}
In this section, we compared the GuSTO algorithm to previous SCP algorithms, TrajOpt \cite{SchulmanDuanEtAl2014} and Mao, et al. \cite{MaoSzmukEtAl2016}, for a 12D free-flying spacecraft robot model, within a cluttered mock-up of the International Space Station. For these dynamics, the state consists of position $\mathbf{r}\in\mathbb{R}^3$, velocity $\mathbf{v}\in\mathbb{R}^3$, the Modified Rodrigues parameters representation of attitude $\mathbf{p}\in\mathbb{R}^3$, and angular velocity $\boldsymbol{\omega}\in\mathbb{R}^3$ \cite{Aoude2007}. Constraints for this system included norm bounds on speed, angular velocity, and control. We modeled the free-flyer robot parameters after the Astrobee robot \cite{SmithBarlowEtAl2016}, details of which can be found at \cite{Astrobee}. The dynamics discretization error reported for an SCP solution was defined as $\sum_{i=1}^{N-1} \|(x^{(i+1)}-x^{(i)})N - f(x^{(i)},u^{(i)})\|_1$. We ran 100 experiments with different start and goal states in the environment shown in Figure \ref{fig:free_flyer_mp_scp}. Each trajectory was initialized with a simple straight line in position space and a geodesic path in rotation space and no control initialization. The results of our simulations are presented in Figure \ref{fig:astrobee3d_results}. For the set of simulations, both GuSTO and TrajOpt successfully returned solutions for 97\% of the trials, though GuSTO on average performed faster and returned higher-quality solutions. Due to the simple initialization often being deep in collision, Mao, et al. had a high failure rate, and failure cases sometimes led to high computation times. 

\begin{figure*}[t!]
\centering
  \begin{subfigure}[t]{0.15\textwidth}
  \centering
  \captionsetup{justification=centering}
    \includegraphics[height=1.5in,trim={0 0.25cm 0 0},clip]{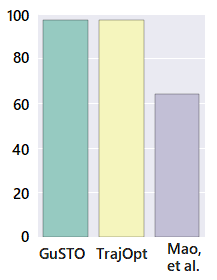}
    \caption{Success percentage}
    \label{fig:astrobee_percent_succ}
  \end{subfigure}
  \qquad
  \begin{subfigure}[t]{0.15\textwidth}
  \centering
  \captionsetup{justification=centering}
    \includegraphics[height=1.5in]{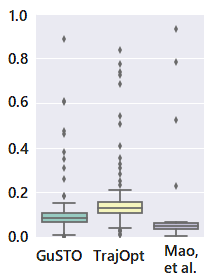}
    \caption{Optimal cost (on success)}
    \label{fig:astrobee_cost}
  \end{subfigure}
  \qquad
  \begin{subfigure}[t]{0.15\textwidth}
    \centering
    \captionsetup{justification=centering}
    \includegraphics[height=1.5in]{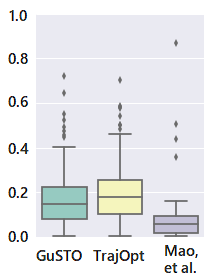}
    \caption{Dynamics discretization error}
    \label{fig:astrobee_dyn_cost}
  \end{subfigure}
  \qquad
    \begin{subfigure}[t]{0.15\textwidth}
    \centering
    \captionsetup{justification=centering}
    \includegraphics[height=1.5in]{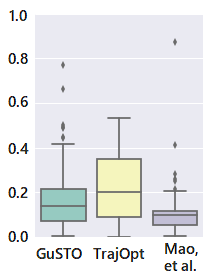}
    \caption{Success computation times}
    \label{fig:astrobee_success_times}
  \end{subfigure}
    \qquad
    \begin{subfigure}[t]{0.15\textwidth}
    \centering
    \captionsetup{justification=centering}
    \includegraphics[height=1.5in]{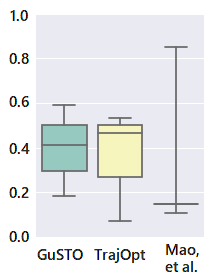}
    \caption{Failure computation times}
    \label{fig:astrobee_failure_times}
  \end{subfigure}
  \caption{Normalized simulation results for the 3D free-flying robot simulation.}
\label{fig:astrobee3d_results}
 \vspace{-15pt}
\end{figure*}

\subsection{Initialization Strategies}
\label{subsec:initialization_strategies}
In practice, when a high-quality initialization trajectory (e.g. dynamically feasible, collision-free, close to a global optimal, etc.) is readily available, it should be used. However, an initial planner (even a coarse one) is often not available and may be expensive to design or time-consuming to run. Thus, we investigate the sensitivity of our approach to initialization, including very simple initialization schemes, again comparing with TrajOpt \cite{SchulmanDuanEtAl2014} and Mao, et al. \cite{MaoSzmukEtAl2016}.

In particular, in this section we ran simulations on an 8D airplane model, having dynamics as in \cite{BeardMcLain2012}. To explore different initialization strategies, we leveraged a recent approach from \cite{SinghChenEtAl2018}. This approach uses a lower-dimensional 4D planning model to generate a path to the goal, which is tracked by a controller to generate a dynamically-feasible trajectory for the full-dimensional system. By planning with tubes that account for model mismatch, the full 8D trajectory is guaranteed to be collision-free.

Using this work, we tested three initialization strategies of increasing quality for the 8D airplane: (1) a simple straight line in the 8D state space, (2) an 8D dynamically-feasible (but possibly in collision) trajectory generated using a controller to track a straight-line initialization in 4D, and (3) an 8D dynamically-feasible and collision-free trajectory recovered from running the full motion planning with model-mismatch tubes in 4D. As illustrated in Figure \ref{fig:airplane}, the problem scenario consists of guiding the airplane to a terminal goal set across a cluttered environment.

For this system, due to the complex coupling in its dynamics, initialization (1) resulted in failure for all three SCP algorithms. The results of initialization (2) can be seen in Fig. \ref{fig:airplane_case2}, where GuSTO found a feasible solution, whereas Mao, et. al. returned a trajectory without satisfying convergence criteria, and TrajOpt resulted in collision. For the highest-quality initialization (3), Mao, et al. did not return a solution, whereas GuSTO and TrajOpt returned feasible trajectories, with run times of 0.55s and 2.67s and cost improvements over the initialization of 55\% and 49\%, respectively.

\subsection{Shooting Method Acceleration}
To investigate the use of dual solutions from our SCP iterations to warm-start shooting methods and accelerate convergence, we ran simulations on a simple 3D Dubin's car and the 12D free-flyer robot, where the free-flyer was placed in an obstacle-filled environment. 

The acceleration technique gave very promising results, as shown in Table \ref{table:shooting}. In practice, running a shooting method to completion, whether to convergence or to a maximum number of iterations, required negligible computation time compared to an SCP iteration ($<1$ ms). Thus, we attempted a shooting method at every iteration of SCP. As shown, using SCP to complete refinement was very time-consuming, whereas the shooting method would converge after just a few SCP iterations, thus reducing computation time for the car and free-flyer models by 52\% and 74\%, respectively.

\newcommand{\addstackgapmod}[1]{\addstackgap[2pt]{#1}}
\newcolumntype{D}{>{\centering\arraybackslash}p{1.15cm}}
\newcommand{\Centerstackmodd}[1]{\addstackgap[2pt]{\Centerstack{#1}}}
\newcolumntype{E}{>{\centering\arraybackslash}p{1.6cm}}

\begin{table}[h!]
\hspace{2.75pt}\begin{tabular}{E | D | D || D | D |}
\cline{2-5}
& \multicolumn{2}{c||}{\addstackgap[2pt]{Dubin's Car}}
& \multicolumn{2}{c|}{\addstackgap[2pt]{Free-flyer Spacecraft}}
\\
\cline{2-5}
& \Centerstackmodd{SCP Only}
& \Centerstackmodd{{SCP +} Shooting}
& \Centerstackmodd{SCP Only}
& \Centerstackmodd{{SCP +} Shooting}
\end{tabular}

\begin{tabular}{| E || D | D || D | D |}
\hline
\addstackgapmod{SCP Iterations} & 12 & 6 & 11 & 3 \\
\hline
\addstackgapmod{Reported Cost} & 19.8 & 19.8 & 6.2 & 6.2 \\
\hline
\addstackgapmod{Running Time} & 94 ms & \textbf{45 ms} & 570 ms & \textbf{146 ms} \\
\hline
\end{tabular}
\caption{Results accelerating convergence by using SCP dual solution to warm-start a shooting method. SCP iterations report are the number required for convergence.}
\label{table:shooting}
\end{table}

\subsection{Hardware Experiments}
We also implemented GuSTO in the Stanford Space Robotics Facility on a free-flyer robot. The robot is a three-DoF, fully holonomic system and is equipped with eight thrusters and a reaction wheel, maneuvering on a frictionless surface. Experiments consisted of a spacecraft navigating a cluttered environment to berth with a capturing spacecraft. Video of the experiments can be found at \url{https://youtu.be/GHehE-If5nY}.

\section{Conclusions}
\label{sec:conclusions}
In this paper, we provided a new generalized approach to solve trajectory optimization problems, based on sequential convex programming. We showed strong theoretical guarantees to ensure broad applicability to many different frameworks in motion planning and trajectory optimization. GuSTO was tested with numerical simulations and experiments showing that more accurate solutions are achieved faster than using recent state-of-the-art SCP solvers.

Future contributions will focus on additional theoretical guarantees. More precisely, we will study higher-order conditions that GuSTO should naturally provide, showing its convergence to more informative points than stationary points, as well as natural extensions of these guarantees to more general manifolds and systems evolving in spaces having a Lie group structure. Finally, GuSTO will be tested on high-DOF systems, such as robot arms and humanoid robots, and we will use the algorithm for hardware experiments on a free-flyer robot in a full $SE(3)$ microgravity environment.

\section*{Acknowledgements}
We would like to thank Brian Coltin, Andrew Symington, and Trey Smith of the Intelligent Robotics Group at NASA Ames Research Center for their discussions during this work, as well as Sumeet Singh, Thomas Lew, and Tariq Zahroof for their assistance on experiment implementation.


\bibliographystyle{IEEEtran}
\bibliography{IEEEabrv,main,ASL_papers}

\iftoggle{ext}{\section*{Appendix A: Proof of Theorem \ref{ref:theoSCP}}
\label{sec:appendix}

\subsection{Pontryagin Maximum Principle}

Our theoretical result provides convergence of SCP procedures towards a quantity satisfying first-order necessary optimality conditions under the Pontryagin Maximum Principle \cite{Pontryagin1987}. Below, we report the Pontryagin Maximum Principle for time-varying problems, which is useful hereafter.

\begin{theorem}[Pontryagin Maximum Principle] \label{ref:theo_PMP}
Let $x$ be an optimal trajectory for (\textbf{OCP}), associated with the control $u$ in $[0,t_f]$. There exist a nonpositive scalar $p^0$ and an absolutely continuous function $p : [0,t_f] \rightarrow \mathbb{R}^n$, called ßadjoint vector, with $(p,p^0) \neq 0$, and such that, almost everywhere in $[0,t_f]$, the following relations hold:
\begin{itemize}
\item \textbf{Adjoint Equations}
\begin{eqnarray} \label{ref:adjointEq}
\begin{cases}
\displaystyle \dot{x}(t) = \frac{\partial H}{\partial p}(t,x(t),p(t),p^0,u(t)) \medskip \\
\displaystyle \dot{p}(t) = -\frac{\partial H}{\partial x}(t,x(t),p(t),p^0,u(t))
\end{cases}
\end{eqnarray}
\item \textbf{Maximality Condition}
\begingroup
\begin{equation} \label{ref:maxCond}
H(t,x(t),p(t),p^0,u(t)) = \underset{v \in U}{\max} \ H(t,x(t),p(t),p^0,v)
\end{equation}
\endgroup
\item \textbf{Transversality Conditions}

If $M_f$ is a submanifold of $M$, locally around $x(t_f)$, then the adjoint vector can be built in order to satisfy
\begin{equation} \label{ref:trasv1}
p(t_f) \quad \perp \quad T_{x(t_f)} M_f
\end{equation}
and, in addition, if the final time $t_f$ is free, one has
\begin{equation} \label{ref:trasv2}
\underset{v \in U}{\max} \ H(t_f,x(t_f),p(t_f),p^0,v) = 0 \quad .
\end{equation}
\end{itemize}
Here, $H(t,x,p,p^0,u) = p \cdot f(t,x,u) + p^0 f^0(t,x,u)$ denotes the Hamiltonian related to (\textbf{OCP}) and the quantity $(x,p,p^0,u)$ is called (Pontryagin) extremal. We say that an extremal is normal if $p^0 \neq 0$ and is abnormal otherwise.
\end{theorem}

It is important to remark that Theorem \ref{ref:theo_PMP} provides more informative multipliers than those given by the Lagrange multiplier rule, because control constraints do not need to be penalized within the cost, the related Hamiltonian is globally maximized, and $(p,p^0)$ are merely continuous functions.

\subsection{Pontryagin Cone Analysis}

We provide the proof of Theorem \ref{ref:theoSCP} for the case of free final time problems, since, for fixed final time problems, the proof is similar but simpler (and quite straightforward, see below). The proof is based on the properties related to Pontryagin cones \cite{Pontryagin1987}. Therefore, we start by providing useful definitions and statements concerning these quantities.

Let $x$ be a feasible trajectory for (\textbf{OCP}), with associated control $u$ in $[0,t_f]$. Throughout the proof, we assume that $t_f$ is a Lebesgue point of $u$. Otherwise, one proceeds using limiting cones as done in \cite{LeeMarkus1967}. For every Lebesgue point $s \in [0,t_f]$ of $u$ and every $v \in U$, we define local variations as
\begingroup
\begin{equation} \label{ref:localVarOCP}
\psi^{s,v}_{x,u} = \left( \begin{array}{c}
\displaystyle f(x(s),v) - f(x(s),u(s)) \\
\displaystyle f^0(x(s),v) - f^0(x(s),u(s))
\end{array} \right) .
\end{equation}
\endgroup
The variation vector $w^{s,v}_{x,u} : [0,t_f] \rightarrow \mathbb{R}^{n+1}$ for (\textbf{OCP}) is the solution of the following variational system
\begin{eqnarray} \label{ref:varOCP}
\begin{cases}
\dot{\psi}(t) = \displaystyle \psi(t) \left( \begin{array}{c}
\displaystyle \frac{\partial f}{\partial x}(x(t),u(t)) \\
\displaystyle \frac{\partial f^0}{\partial x}(x(t),u(t))
\end{array} \right) \bigskip \\
\psi(s) = \psi^{s,v}_{x,u}
\end{cases} .
\end{eqnarray}
At this step, for every $t \in [0,t_f]$, we define the Pontryagin cone $K_{x,u}(t)$ at $t$ for $(x,u)$ related to (\textbf{OCP}) to be the smallest closed convex cone containing $w^{s,v}_{x,u}(t)$ for every $0 < s < t$ Lebesgue point of $u$ and every $v \in U$. Arguing by contradiction \cite{AgrachevSachkov2004,HaberkornTrelat2011}, the Pontryagin Maximum Principle states that, if $(t_f,x,u)$ is optimal for (\textbf{OCP}), then there exists a nontrivial couple $(p_f,p^0) \in \mathbb{R}^{n+1}$ satisfying
\begin{eqnarray} \label{ref:variationalOCP}
\begin{cases}
p_f \quad \perp \quad T_{x(t_f)} M_f \quad , \quad p^0 \le 0 \medskip \\
(p_f,p^0) \cdot w \le 0 \quad , \quad \forall \ w \in K_{x,u}(t_f) \medskip \\
\underset{v \in U}{\max} \ H(x(t_f),p_f,p^0,v) = 0
\end{cases} .
\end{eqnarray}
Relations \eqref{ref:adjointEq}-\eqref{ref:trasv2} derive from \eqref{ref:variationalOCP}. In particular, a tuple $(x,p,p^0,u)$ is a Pontryagin extremal for (\textbf{OCP}) iff the nontrivial couple $(p(t_f),p^0) \in \mathbb{R}^{n+1}$ satisfies \eqref{ref:variationalOCP}. However, a Pontryagin extremal is not necessarily a solution for (\textbf{OCP}).

\addtolength{\textheight}{-2.5cm}   

Now, consider controls $u_{k}$ and $u_{k+1}$, solutions of (\textbf{LOCP})$_{k}$ and of (\textbf{LOCP})$_{k+1}$ with final times $t^k_f$ and $t^{k+1}_f$, respectively. If necessary and without loss of generality, thanks to Assumption $(A_3)$ we can continuously extend these controls to be constant in $[t^k_f,b]$ and $[t^{k+1}_f,b]$, respectively. We apply the same procedure to trajectories $x_{k}$ and $x_{k+1}$. Therefore, for every iteration $k \ge 1$, $u_{k}$, $u_{k+1}$ are continuous functions in $[0,b]$ and for every $s \in [0,b]$ and every $v \in U$, we are able to define local variations for (\textbf{LOCP})$_{k+1}$ as
\begingroup
\small
\begin{equation} \label{ref:localVarLOCP}
\psi^{s,v}_{k+1} = \left( \begin{array}{c}
\displaystyle f_{k+1}(s,x_{k+1}(s),v) - f_{k+1}(s,x_{k+1}(s),u_{k+1}(s)) \\
\displaystyle f^0_{k+1}(s,x_{k+1}(s),v) - f^0_{k+1}(s,x_{k+1}(s),u_{k+1}(s))
\end{array} \right) 
\end{equation}
\endgroup
and related variation vectors $w^{s,v}_{k+1} : [0,b] \rightarrow \mathbb{R}^{n+1}$ as the solutions of the following variational system:
\begin{eqnarray} \label{ref:varLOCP}
\begin{cases}
\dot{\psi}(t) = \displaystyle \psi(t) \left( \begin{array}{c}
\displaystyle \frac{\partial f_{k+1}}{\partial x}(t,x_{k+1}(t),u_{k+1}(t)) \\
\displaystyle \frac{\partial f^0_{k+1}}{\partial x}(t,x_{k+1}(t),u_{k+1}(t))
\end{array} \right) \bigskip \\
\psi(s) = \psi^{s,v}_{k+1}
\end{cases} .
\end{eqnarray}
Thus, from the above and due to the optimality of $(t^{k+1}_f,x_{k+1},u_{k+1})$, the Pontryagin Maximum Principle states that, for every iteration $k$ of SCP, there exists a nontrivial couple $(p_{k+1},p^0_{k+1}) \in C^0([0,b],\mathbb{R}^n) \times \mathbb{R}$ satisfying
\begingroup
\small
\begin{eqnarray} \label{ref:variationalLOCP}
\begin{cases}
p_{k+1}(t^{k+1}_f) \quad \perp \quad T_{x_{k+1}(t^{k+1}_f)} M_f \quad , \quad p^0_{k+1} \le 0 \medskip \\
(p_{k+1}(t^{k+1}_f),p^0_{k+1}) \cdot w \le 0 \quad , \quad \forall \ w \in K_{k+1}(t^{k+1}_f) \medskip \\
\underset{v \in U}{\max} \ H_{k+1}(t^{k+1}_f,x_{k+1}(t^{k+1}_f),p_{k+1}(t^{k+1}_f),p^0_{k+1},v) = 0
\end{cases}
\end{eqnarray}
\endgroup
where the Pontryagin cone $K_{k+1}(t)$ at $t$ for (\textbf{LOCP})$_{k+1}$ is defined as above, by substituting \eqref{ref:localVarOCP}-\eqref{ref:varOCP} with \eqref{ref:localVarLOCP}-\eqref{ref:varLOCP}, and $H_{k+1}$ is the Hamiltonian related to problem (\textbf{LOCP})$_{k+1}$.

We now prove Theorem \ref{ref:theoSCP} in two main steps:

\subsubsection{Convergence of Trajectories and Controls}

First, consider the sequence of final times $(t^{k}_f)_{k \in \mathbb{N}}$. Thanks to Assumption $(A_3)$, there exists $\tilde t_f \in [0,b]$ such that, up to some subsequence, $(t^{k}_f)_{k \in \mathbb{N}}$ converges to $\tilde t_f$. As discussed previously, from now on, we consider every couple $(x_k,u_k)$ to be continuously defined in the time interval $[0,b]$.

Next, consider the sequence $(u_k)_{k \in \mathbb{N}} \subseteq L^{\infty}([0,b],U)$. Thanks to Assumption $(A_1)$, $(u_k)_{k \in \mathbb{N}}$ is bounded in $L^2([0,b],\mathbb{R}^m)$. Moreover, the subset $L^2([0,b],U)$ is closed and convex in $L^2([0,b],\mathbb{R}^m)$ for the strong topology, and then also for the weak topology \cite{Brezis2011}. Thanks to Assumption $(A_1)$ and reflexive properties for $L^2$, there exists $\tilde u \in L^{\infty}([0,b],U)$ such that, up to some subsequence, $(u_k)_{k \in \mathbb{N}}$ converges to $\tilde u$ for the weak topology of $L^2$ \cite{Brezis2011}.

Finally, we focus on the sequence $(x_k)_{k \in \mathbb{N}} \subseteq C^0([0,b],\mathbb{R}^n)$. It is clear that Assumptions $(A_1)$ and $(A_2)$ provide that both $(x_k)_{k \in \mathbb{N}}$ and $(\dot{x}_k)_{k \in \mathbb{N}}$ are bounded in $L^2([0,b],\mathbb{R}^n)$. Therefore, $(x_k)_{k \in \mathbb{N}}$ is bounded in the Sobolev space $H^1([0,b],\mathbb{R}^n)$. From reflexive properties, it follows that there exists $\tilde x \in H^1([0,b],\mathbb{R}^n)$ such that, up to some subsequence, $(x_k)_{k \in \mathbb{N}}$ converges to $\tilde x$ for the weak topology of $H^1$. Furthermore, since the inclusion $H^1 \xhookrightarrow{} C^0$ is compact, $(x_k)_{k \in \mathbb{N}}$ converges to $\tilde x \in C^0([0,b],\mathbb{R}^n)$ for the strong topology of $C^0$ \cite{Brezis2011}

For every integer $k$, $(x_{k+1},u_{k+1})$ is feasible for (\textbf{LOCP})$_{k+1}$, and therefore (after the obvious extensions),
$$
x_{k+1}(t) = \bar{x}_0 + \int_{0}^{t} f_{k+1}(s,x_{k+1}(s),u_{k+1}(s)) \; \mathrm{d}s \ , \ t \in [0,b] .
$$
From this, by exploiting Assumptions $(A_1)$, $(A_2)$, and the previous convergences, it follows that $(\tilde x,\tilde u)$ is feasible for problem (\textbf{OCP}) (note that $\tilde x(\tilde t_f) = \underset{k \rightarrow \infty}{\lim} x_k(t^k_f) \in M_f$, since, up to some subsequence, the limit $\underset{k \rightarrow \infty}{\lim} x_k(t^k_f)$ exists thanks to the compactness of $M_f$, see Assumption $(A_1)$).

\subsubsection{Convergence of Multipliers}

We now discuss the convergence to a Pontryagin extremal for (\textbf{OCP}). Assumption $(A_4)$ proves crucial to establishing the following Lemma:

\begin{lemma} \label{ref:lemmaLebesgue}
Suppose that Assumption $(A_4)$ holds. For every $s \in (0,\tilde t_f)$ Lebesgue point of $\tilde u$, there exists a sequence $(s_k)_{k \in \mathbb{N}} \subseteq [s,\tilde t_f)$, for which $s_k$ is a Lebesgue point of $u_k$ and of $u_{k+1}$, such that
$$
u_k(s_k) \rightarrow \tilde u(s) \quad , \quad u_{k+1}(s_k) \rightarrow \tilde u(s) \quad , \quad s_k \rightarrow s
$$
as $k$ tends to infinity.
\end{lemma}

\begin{proof} We denote
$$
h_k(t) = (u_k(t),u_{k+1}(t)) \ , \ h(t) = (u(t),u(t)).
$$
Let us prove that, for every $s \in (0,\tilde t_f)$ Lebesgue point of $h(\cdot)$ and for every $\beta > 0$, $\alpha_s > 0$ (such that $s+\alpha_s < \tilde t_f$), there exists $\gamma_{s,\alpha_s,\beta} > 0$ such that, for every $k \in \mathbb{N}$ satisfying $1/k \in (0,\gamma_{s,\alpha_s,\beta})$, there exists a $s_k \in [s,s+\alpha_s]$ Lebesgue point of $h_k(\cdot)$ for which $\| h_k(s_k) - h(s) \| < \beta$.

By contradiction, suppose that there exists $s \in (0,\tilde t_f)$, a Lebesgue point of $h(\cdot)$, and $\beta > 0$, $\alpha_{s} > 0$ (with $s+\alpha_s < \tilde t_f$) such that, for every $\gamma > 0$, there exists $k \in \mathbb{N}$ with $1/k \in (0,\gamma)$ and $i_k \in \{ 1,\dots,m \}$ for which, for $t \in [s,s+\alpha_{s}]$ Lebesgue point of $h_k(\cdot)$, it holds that $|h^{i_k}_k(t) - h^{i_k}(s)| \geq \beta$.

From the previous convergence results, the family $(h_k(\cdot))_{k \in \mathbb{N}}$ converges to $h(\cdot)$ in $L^2$ for the weak topology. Therefore, for every $0 < \delta \le 1$, there exists an integer $k_{\delta}$ such that, for every $k \geq k_{\delta}$, it holds that
$$
\frac{1}{\delta \alpha_s} \Big| \int_s^{s+\delta \alpha_s} h^i_k(t) \ \mathrm{d}t - \int_s^{s+\delta \alpha_s} h^i(t) \ \mathrm{d}t \Big| < \frac{\beta}{3} \quad 
$$
for every $i \in \{1,\dots,m\}$. We exploit this fact to bound $|h^{i_k}_k(t) - h^{i_k}(s)|$ by $\beta$. First, since $s$ is a Lebesgue point of $h(\cdot)$, there exists $0 < \delta_{s,\alpha_s} \le 1$ such that
$$
\Big| h^i(s) - \frac{1}{\delta_{s,\alpha_s} \alpha_s} \int_s^{s+\delta_{s,\alpha_s} \alpha_s} h^i(t) \ \mathrm{d}t \Big| < \frac{\beta}{3}
$$
for every $i \in \{1,\dots,m\}$. On the other hand, from what was said previously, there exists an integer $k_{\delta_{s,\alpha_s}}$ such that
$$
\frac{1}{\delta_{s,\alpha_s} \alpha_s} \Big| \int_s^{s+\delta_{s,\alpha_s} \alpha_s} h^i_k(t) \ \mathrm{d}t - \int_s^{s+\delta_{s,\alpha_s} \alpha_s} h^i(t) \ \mathrm{d}t \Big| < \frac{\beta}{3}
$$
for every $k \ge k_{\delta_{s,\alpha_s}}$ and every $i \in \{1,\dots,m\}$. Finally, by Assumption $(A_4)$, we have that $h_k(\cdot)$ is continuous for $k \in \mathbb{N}$, and then, for every $k \ge k_{\delta_{s,\alpha_s}}$ and every $i \in \{1,\dots,m\}$, there exists $t_{k,i} \in [s,s+\delta_{s,\alpha_s} \alpha_s] \subseteq [s,s+\alpha_s]$ such that
$$
\Big| h^i_k(t_{k,i}) - \frac{1}{\delta_{s,\alpha_s} \alpha_s} \int_s^{s+\delta_{s,\alpha_s} \alpha_s} h^i_k(t) \ \mathrm{d}t \Big| < \frac{\beta}{3} \quad .
$$
Resuming, for every $k \ge k_{\delta_{s,\alpha_s}}$ and $i \in \{ 1,\dots,m \}$ there exists a $t_{k,i} \in [s,s+\alpha_s]$ Lebesgue point of $h_k(\cdot)$ (by continuity) such that $|h^i_k(t_{k,i}) - h^i(s)| < \beta$, a contradiction.
\end{proof}

Lemma \ref{ref:lemmaLebesgue} represents the main tool to prove the convergence of Pontryagin cones, provided by the following lemma:

\begin{lemma} \label{ref:lemmaCone}
For every $w \in K_{\tilde x,\tilde u}(\tilde t_f)$, $k \in \mathbb{N}$, there exists $w_{k} \in K_{k}(t^k_f)$ such that $w_k \rightarrow w$ as $k$ tends to infinity.
\end{lemma}
\begin{proof}Without loss of generality, we may assume that $w = w^{s,v}_{\tilde x,\tilde u}(\tilde t_f)$, where $v \in U$ and $0<s<\tilde t_f$ is a Lebesgue point of $\tilde u$ (see \cite[Lemma 7.8]{Bonalli2018} for technical details).

From Lemma \ref{ref:lemmaLebesgue}, there exists a family $(s_k)_{k \in \mathbb{N}} \subseteq [s,\tilde t_f)$, which are Lebesgue points of $u_k$ and of $u_{k+1}$, such that
$$
u_k(s_k) \rightarrow \tilde u(s) \quad , \quad u_{k+1}(s_k) \rightarrow \tilde u(s) \quad , \quad s_k \rightarrow s
$$
as soon as $k$ tends to infinity. This allows us to consider $w^{s_k,v}_{k+1}$, solutions of system \eqref{ref:varLOCP} with initial state given by \eqref{ref:localVarLOCP} at $s_k$.

From the previous convergences, it is clear that $(\psi^{s_k,v}_{k+1})_{k \in \mathbb{N}}$ converges to $\psi^{s,v}_{\tilde x,\tilde u} = w^{s,v}_{\tilde x,\tilde u}(s)$ as soon as $k$ tends to infinity. Moreover, since $(\Delta_k,\omega_k)_{k \in \mathbb{N}} \subseteq [0,\Delta_0] \times [\omega_0,\omega_{\max}]$ is bounded, we have that up to some subsequence, it converges to some point $(\tilde \Delta,\tilde \omega) \in [0,\Delta_0] \times [\omega_0,\omega_{\max}]$ satisfying either $\tilde \Delta = 0$ or $\tilde \Delta > 0$. In both cases, again from the previous convergences, the dynamics of system \eqref{ref:varLOCP} converge to the dynamics of system \eqref{ref:varOCP} for the weak topology of $L^2$. Summing up, by the continuous dependence w.r.t. initial state and weakly w.r.t. controls for dynamical systems , the sequence $(w_{k+1})_{k \in \mathbb{N}} = (w^{s_k,v}_{k+1}(t^k_f))_{k \in \mathbb{N}}$ satisfies $w_{k} \in K_{k}(t^k_f)$ and converges to $w$ as $k$ tends to infinity.
\end{proof}

We are now able to conclude the proof of Theorem \ref{ref:theoSCP}.

For every integer $k \ge 1$, consider the nontrivial couple $(p_{k},p^0_{k}) \in C^0([0,t^k_f],\mathbb{R}^n) \times \mathbb{R}$, provided by Theorem \ref{ref:theo_PMP}, related to some optimal solution $(t^k_f,x_k,u_k)$ for (\textbf{LOCP})$_k$. In particular, $(p_{k}(t^k_f),p^0_{k}) \neq 0$. Therefore, up to normalization, we can assume that $\| (p_k(t^k_f),p^0_k) \| = 1$, for every $k \in \mathbb{N} \setminus \{ 0 \}$. We infer that, up to some subsequence, there exists a point $(\tilde p_f,\tilde p^0) \in S^n$ (in particular, $(\tilde p_f,\tilde p^0) \neq 0$) satisfying $(p_k(t^k_f),p^0_k) \rightarrow (\tilde p_f,\tilde p^0)$ as $k$ tends to infinity.

Now, take any $w \in K_{\tilde x,\tilde u}(\tilde t_f)$. Thanks to Lemma \ref{ref:lemmaCone}, there exists a sequence $(w_k)_{k \in \mathbb{N}}$, such that $w_k \in K_k(t^k_f)$, which converges to $w$ as soon as $k$ tends to infinity. By continuity, from \eqref{ref:variationalLOCP} it follows that $\tilde p_f \perp T_{\tilde x(\tilde t_f)} M_f$ and $(\tilde p_f,\tilde p^0) \cdot w \le 0$. Moreover, since $((\Delta_k,\omega_k))_{k \in \mathbb{N}} \subseteq [0,\Delta_0] \times [\omega_0,\omega_{\max}]$ is bounded, up to some subsequence, it converges to some point $(\tilde \Delta,\tilde \omega) \in [0,\Delta_0] \times [\omega_0,\omega_{\max}]$ satisfying either $\tilde \Delta = 0$ or $\tilde \Delta > 0$. In both cases, it is not difficult to prove that
$$
\underset{k \rightarrow \infty}{\lim} |H(\tilde x(\tilde t_f),\tilde p(\tilde t_f),\tilde p^0,v) - H_{k}(t^k_f,x_k(t^f_k),p_k(t^k_f),p^0_k,v)| = 0
$$
uniformly with respect to $v \in U$. Therefore, since $w \in K_{\tilde x,\tilde u}(t_f)$ is arbitrary, $(\tilde p_f,\tilde p^0)$ satisfies relations \eqref{ref:variationalOCP} for $(x,u) = (\tilde x,\tilde u)$, and then, denoting by $\tilde p$ the solution of
\begin{eqnarray*}
\begin{cases}
\displaystyle \dot{p}(t) = -(p(t),\tilde p^0) \left( \begin{array}{c}
\displaystyle \frac{\partial f}{\partial x}(\tilde x(t),\tilde u(t)) \\
\displaystyle \frac{\partial f^0}{\partial x}(\tilde x(t),\tilde u(t))
\end{array} \right) \bigskip \\
p(t_f) = \tilde p_f
\end{cases}
\end{eqnarray*}
the quantity $(\tilde x,\tilde p,\tilde p^0,\tilde u)$ represents a Pontryagin extremal for problem (\textbf{OCP}). In particular, thanks to the previous convergences and the continuous dependence w.r.t. initial state and weakly w.r.t. controls for dynamical systems, we have that up to some subsequence, $(p_{k})_{k \in \mathbb{N}}$ converges to $\tilde p$ for the strong topology of $C^0$, as $k \rightarrow \infty$.

This concludes the proof of Theorem \ref{ref:theoSCP}.

\begin{remark}
In formulations (\textbf{LOCP})$_{k}$, we linearize the terms depending on the state. However, in the case that convex functions of the state appear within the cost, both our numerical scheme and our theoretical result still hold even if these convex terms are not linearized (the proof of this fact exactly retraces  the proof of our convergence result).
\end{remark}

\section*{Appendix B: Additional Experimental Details}
\label{sec:experimental_details}
\setcounter{subsection}{0}
\subsection{Free-Flying Spacecraft Robot Model}
The 12D state for the free-flying spacecraft robot model consists of position $\mathbf{r}\in\mathbb{R}^3$, velocity $\mathbf{v}\in\mathbb{R}^3$, the Modified Rodrigues parameters representation of attitude $\mathbf{p}\in\mathbb{R}^3$, and angular velocity $\boldsymbol{\omega}\in\mathbb{R}^3$ \cite{Aoude2007,Shuster1993}, and the control variables are the force $\mathbf{F}\in\mathbb{R}^3$ and moment $\mathbf{M}\in\mathbb{R}^3$.  The continuous-time dynamics are given by
\begin{align*}
\begin{pmatrix} \dot{\mathbf{r}}\\ \dot{\mathbf{v}}\\ \dot{\mathbf{p}}\\ \dot{\boldsymbol{\omega}} \end{pmatrix} = \begin{pmatrix}
\mathbf{v}\\ \mathbf{F} / m \\ \frac{1}{4}((1-|\mathbf{p}|^2)\boldsymbol{\omega} - 2 \boldsymbol{\omega} \times \mathbf{p} + 2(\boldsymbol{\omega}\cdot\mathbf{p})\mathbf{p}) \\
J^{-1}(\mathbf{M}-\boldsymbol{\omega}\times J \boldsymbol{\omega})
\end{pmatrix}
\end{align*}
where $m$ and $J$ are the robot mass and inertia tensor, respectively. State constraints for this system include norm bounds for velocity and angular velocity, as well as norm bound control constraints for the force and moment.

\subsection{Airplane Model}
The state for the 8D airplane model consists of the position $x,y,z$, course angle $\psi$, airspeed $v$, flight path angle $\gamma$, roll angle $\phi$, and angle-of-attack $\alpha$ \cite{BeardMcLain2012}. The control inputs consist of longitudinal acceleration $u_a$, roll rate $u_{\dot{\phi}}$, and pitch rate $u_{\dot{\alpha}}$. The continuous-time dynamics are given by
\begin{align*}
\begin{pmatrix}\dot{x}\\ \dot{y}\\ \dot{z}\\ \dot{\psi}\\ \dot{v}\\ \dot{\gamma}\\ \dot{\phi}\\ \dot{\alpha}\end{pmatrix} = 
\begin{pmatrix}
v \cos\psi \cos\gamma \\
v \sin \psi \cos\gamma \\
v \sin\gamma \\
-F_{\mathrm{lift}}(v,\alpha)\sin\phi / (mv\cos\gamma) \\
u_a - F_{\mathrm{drag}}(v,\alpha) / m - g\sin\gamma \\
F_{\mathrm{lift}}(v,\alpha)\cos\phi / (mv) - g\cos\gamma / v \\
u_{\dot{\phi}} \\ 
u_{\dot{\alpha}}
\end{pmatrix}
\end{align*}
where $m$ is the airplane mass and $g$ is gravitational acceleration. A flat-plate airfoil model is used for calculating the lift force $F_{\mathrm{lift}}(v,\alpha)=\pi\rho A v^2 \alpha$ and drag force $F_{\mathrm{drag}}(v,\alpha) = \rho A v^2 (C_{D_0}+4\pi K\alpha^2)$, for air density $\rho$, wing area $A$, drag coefficient $C_{D_0}$, and the induced drag factor $K$. Constraints for this system consist of box constraints on states $\psi$, $v$, $\gamma$, and $\phi$, as well as on controls $u_a$, $u_{\dot{\phi}}$, and $u_{\dot{\alpha}}$.

\subsection{Dubins Car Model}
The Dubin's car model used is a simple three-dimensional kinematic model consisting of positions $x$ and $y$, orientation $\theta$, and steering control $u$ \cite{LaValle2006}. The continuous-time dynamics are given by:
\begin{align*}
\begin{pmatrix} \dot{x}\\ \dot{y}\\ \dot{\theta} \end{pmatrix} = \begin{pmatrix}
v \cos\theta \\ v \sin\theta \\ k u\\
\end{pmatrix}
\end{align*}
where $v$ is the constant speed of the car and $k$ is a curvature constant.}

\end{document}